\documentclass[12pt]{amsart}
\usepackage{fullpage}

\newtheorem{theorem}{Theorem}

\newtheorem{lemma}[theorem]{Lemma}
\newtheorem{proposition}[theorem]{Proposition}

\theoremstyle{remark}
\newtheorem{remark}{Remark}

\begin{document}

\title{Cubefree words with many squares}

\author{James Currie and Narad Rampersad}

\address{Department of Mathematics and Statistics \\
University of Winnipeg \\
515 Portage Avenue \\
Winnipeg, Manitoba R3B 2E9 (Canada)}

\email{\{j.currie,n.rampersad\}@uwinnipeg.ca}

\thanks{The first author is supported by an NSERC Discovery Grant.}
\thanks{The second author is supported by an NSERC Postdoctoral
Fellowship.}


\date{\today}

\begin{abstract}
We construct infinite cubefree binary words containing exponentially
many distinct squares of length $n$.  We also show that for every
positive integer $n$, there is a cubefree binary square of length $2n$.
\end{abstract}

\maketitle

\section{Introduction}

A \emph{square} is a non-empty word of the form $xx$, and a \emph{cube} is a
non-empty word of the form $xxx$.  An overlap is a word of the form
$axaxa$, where $a$ is a letter and $x$ is a word (possibly empty).
A word is \emph{squarefree} (resp. \emph{cubefree}, \emph{overlap-free})
if none of its factors are squares (resp. cubes, overlaps).  For further
background material concerning combinatorics on words we refer the reader to
\cite{AS03}.

It is well-known that there exist infinite squarefree words over a
ternary alphabet and infinite overlap-free words over a binary alphabet.
Clearly, any overlap-free word is also cubefree.  Any infinite cubefree
binary word must contain squares; however, Dekking \cite{Dek76} proved
that there exists an infinite cubefree binary word containing no squares $xx$
where the length of $x$ is greater than $3$ (see also \cite{RSW05,Sha04}).
In this paper we consider instead the existence of infinite cubefree binary
words with many distinct squares.

Most known constructions of infinite cubefree words involve the iteration
of a morphism.  Words constructed in this manner are often refered to as
\emph{infinite D0L words}.  Ehrenfeucht and Rozenberg \cite{ER81,ER83a,ER83b}
proved several results concerning the factor complexity of infinite D0L
words.  They showed that any squarefree or cubefree D0L word has $O(n \log n)$
factors of length $n$.  Thus, an infinite cubefree D0L word
cannot have many distinct square factors.  By constrast, we show here how to
construct infinite cubefree binary words containing exponentially
many distinct squares of length $n$.

Other work related to the problems considered here include
\cite{AC04,CR08,CRS06}.

Let $\mu$ denote the \emph{Thue--Morse morphism}:
i.e., the morphism that maps $0 \to 01$ and $1 \to 10$.
The \emph{Thue--Morse word} is the infinite word
\[
{\bf t} = 011010011001011010010110 \cdots
\]
obtained by iteratively applying $\mu$ to the word $0$.
The Thue--Morse word is well-known to be overlap-free, and hence, a fortiori,
cubefree \cite{Thu12}.  The squares occurring in the Thue--Morse word were
characterized by Pansiot \cite{Pan81} and Brlek \cite{Brl89}
as follows.  Define sets $A = \{00,11,010010,101101\}$ and
\[
\mathcal{A} = \bigcup_{k \geq 0}\mu^k(A).
\]
The set $\mathcal{A}$ is the set of squares appearing in the Thue--Morse word.

Shelton and Soni \cite{SS85} characterized the overlap-free squares
(the result is also attributed to Thue by Berstel \cite{Ber92}),
as being the conjugates of the words in $\mathcal{A}$.  (A \emph{conjugate}
of $x$ is a word $y$ such that $x = uv$ and $y = vu$ for some $u,v$.)
Currie and Rampersad \cite{CR08} showed that the conjugates of the words in
$\mathcal{A}$ are also precisely the $7/3$-power-free squares.
Thus, there are only $7/3$-power-free squares of length $2n$ when $n$
is a power of $2$, or $3$ times a power of $2$.  By contrast, we show that
there are cubefree binary squares of length $2n$ for every positive integer
$n$.  We use this result to construct infinite cubefree binary words
containing exponentially many distinct squares.

\section{Main results}

The main results of this paper are the following two theorems.

\begin{theorem}
\label{cbfsqs}
Let $n$ be a positive integer.  There exists a cubefree binary
square of length $2n$.
\end{theorem}

\begin{theorem}
\label{main}
There exists an infinite cubefree binary word containing exponentially
many distinct squares of length $n$.
\end{theorem}

We first establish some preliminary results.

\begin{lemma}
\label{findx}
The Thue--Morse word contains a factor of the form $x =
1001x^{\prime\prime}=x'1001$ of every positive even length $n\ne 2,6$.
\end{lemma}

\begin{proof}
Aberkane and Currie \cite[Lemma~4]{AC04} proved that for every
integer $m \geq 6$, the Thue--Morse word contains a factor of length $m$
of the form $10y10$.  Then the Thue--Morse word also contains
the factor $\mu(10y10) = 1001\mu(y)1001$, which has length $2m$.  Finally,
we observe that $10011001$ and $1001101001$ are factors of the Thue--Morse
word of lengths $8$ and $10$ respectively.
\end{proof}

\begin{lemma}\label{overlap-free}
If $y$ is overlap-free and $ayb$ is a cube of period $p$, then $p \le |ab|$.
\end{lemma}

\begin{proof}
Otherwise deleting $a$ and $b$ removes less than a full period from
$ayb$, leaving an overlap.
\end{proof}

\begin{lemma}\label{occurs twice}
If $z$ is a factor of $yyy$ where $|y|=p$ and $|z|\le p+1$, then
there are two occurrences of $z$ in $yyy$.
\end{lemma}

\begin{proof}
Certainly if $z$ is a factor of $yy$ it occurs twice in $yyy$. If
z is a factor of $yyy$ but not of $yy$, then $z$ must span the central
$y$ of $yyy$ and a bit more on both ends, giving $z$ a length of $p +
2$ or more.
\end{proof}

\begin{theorem}\label{odd n}
Let $x$ be a factor of the Thue--Morse word of the form
$x = 1001x^{\prime\prime}=x'1001$. Then the word $x0x0$is cubefree.
\end{theorem}

\begin{remark}\label{01010}
Word $01010$ occurs exactly once in $x0x0$. (Note that this word is an
overlap, and hence not a factor of the Thue--Morse word.)
\end{remark}

\begin{proof}[Proof of Theorem~\ref{odd n}]
Suppose $yyy$ is a cube in $x0x0$ with $|y|=p>0$.

\noindent{\bf Case 1:} {\it Period $p \ge 4$.}
By Lemma~\ref{occurs twice} and Remark~\ref{01010}, word $01010$ is not
a factor of $yyy$. We have two possibilities:

{\bf Case 1a:} {\it Cube $yyy$ is a factor of $x'100101$.}  This is
impossible by Lemma~\ref{overlap-free}, since $x'1001$ is
overlap-free, $|01|=2$, and $p\ge 4>2.$

{\bf Case 1b:} {\it Cube $yyy$ is a factor of $101001x^{\prime\prime}0$.}
This is again impossible by Lemma~\ref{overlap-free},
since $1001x^{\prime\prime}$ is overlap-free.

\noindent {\bf Case 2:} {\it Period $p \le 3.$} If $01010$ is a factor
of $yyy$, then one of $001010$ and $010100$ is a factor. However, neither
of these has period 1, 2 or 3; this is impossible. We conclude that
$01010$ is not a factor of $yyy$. This gives a similar case breakdown as
in Case~1.

{\bf Case 2a:} {\it Cube $yyy$ is a factor of $x'100101$.}

\hspace{.25in}{\bf Case 2ai:} {\it Cube $yyy$ is a suffix of
$x'100101$.} In this case, $p \le 2$ by Lemma~\ref{overlap-free},
since $x'1001$ is overlap-free. However, the longest suffix of
$x'100101$ of period 1 or 2 is $0101$, which is cubefree.

\hspace{.25in}{\bf Case 2aii:} {\it Cube $yyy$ is a suffix of
$x'10010$.} This forces $p = 1$, which is impossible.

{\bf Case 2b:} {\it Cube $yyy$ is a factor of $101001x^{\prime\prime}0$.}

\hspace{.25in}{\bf Case 2bi:} {\it Cube $yyy$ is a prefix of
$101001x^{\prime\prime}0$ or of $01001x^{\prime\prime}0$.} Since
$|yyy|=3p\le 9\le |01001x^{\prime\prime}|$, $yyy$ is a factor of
$101001x^{\prime\prime}$. This is symmetrical to Case~2a.

\hspace{.25in}{\bf Case 2bii:} {\it Cube $yyy$ is a factor of
$1001x^{\prime\prime}0 = x0$.} This is impossible by Case~2a.
\end{proof}

\begin{theorem}\label{even}
Let $x$ be a factor of the Thue--Morse word of the form $x =
1001x^{\prime\prime}=x'1001$. Then the word $x101100x101100$ is
cubefree.
\end{theorem}

\begin{remark}\label{00100}
Word $00100$ occurs exactly once in $x101100x101100$. Word $11011$
occurs exactly twice.
\end{remark}

\begin{proof}[Proof of Theorem~\ref{even}]
Suppose $yyy$ is a cube in $x101100x101100$ with $|y|=p>0$.

\noindent{\bf Case 1:} {\it Period $p \ge 4$.}
By Lemma~\ref{occurs twice} and Remark~\ref{00100}, word $00100$ is not
a factor of $yyy$. We have two possibilities:

{\bf Case 1a:} {\it Cube $yyy$ is a factor of $x10110010$.}  Word
$x10110010$ contains $11011$ as a factor exactly once. By
Lemma~\ref{occurs twice} and Remark~\ref{00100}, there are two
possibilities:

\hspace{.25in}{\bf Case 1ai:} {\it Cube $yyy$ is contained in $x101$.}
In this case, $p \le 3$ by Lemma~\ref{overlap-free}, since $x$ is
overlap-free. This is a contradiction.

\hspace{.25in}{\bf Case 1aii:} {\it Cube $yyy$ is contained in
$10110010$.} This is clearly impossible.

{\bf Case 1b:} {\it Cube $yyy$ is a factor of $0x101100$.}  Again,
word $0x101100$ contains $11011$ as a factor exactly once. Therefore,
either $yyy$ is contained in $101100$ or in $0x101$. The first
alternative evidently is impossible, while the second is ruled out by
Lemma~\ref{overlap-free}.

\noindent {\bf Case 2:} {\it Period $p \le 3.$} If $00100$ is a factor
of $yyy$, then we must have $p=3$, since $00100$ does not have period 1
or 2. However, in $x101100x101100$, the maximal factor of period 3
containing $00100$ is $1001001$, which is not a cube. We conclude that
$00100$ is not a factor of $yyy$. This gives a similar case breakdown to
Case~1:

{\bf Case 2a:} {\it Cube $yyy$ is a factor of $x10110010$.} By
Lemma~\ref{overlap-free} the word $x10$ must be cubefree. Therefore,
$yyy$ must be a suffix of one of these words:
\begin{eqnarray*}
w_8&=&x'100110110010\\
w_7&=&x'10011011001\\
w_6&=&x'1001101100\\
w_5&=&x'100110110\\
w_4&=&x'10011011\\
w_3&=&x'1001101
\end{eqnarray*}
None of the $w_n$ ends in a cube of period 1, 2 or 3. (In the case of
words $w_4$, $w_3$, the longest suffixes of period 3 have lengths 6
and 5 respectively.) It follows that $yyy$ is not a suffix of any of
the $w_n$, and this case does not occur.

{\bf Case 2b:} Cube $yyy$ is a factor of $0x101100$. Since
$|yyy|=3p\le 9\le|0x|$, $yyy$ is a factor of $0x$ or of $x101100$. The
first possibility was ruled out in Theorem~\ref{odd n}, and the second
in Case~2a.
\end{proof}

Theorems~\ref{odd n} and \ref{even} together establish Theorem~\ref{cbfsqs}.
Next we show that the number of cubefree binary squares of length $n$
grows exponentially.

\begin{proposition}
\label{exp_many}
There exist exponentially many cubefree binary squares of length $n$.
\end{proposition}

\begin{proof}
Let $m$ be a positive integer and let $xx$ be a cubefree binary square
of length $2m$ over $\{0,1\}$.  Suppose that $0$ occurs at least as often as
$1$ in $x$.  Construct a new cubefree square $yy$ over $\{0,1,2\}$,
where $y$ is obtained from $x$ by arbitrarily replacing some of the $0$'s
in $x$ by $2$'s.  There are at least $2^{m/2}$ such squares $yy$
of length $2m$.

Let $h$ be the morphism
\begin{eqnarray*}
0 & \to & 001011 \\
1 & \to & 001101 \\
2 & \to & 011001.
\end{eqnarray*}
Brandenburg \cite[Theorem~6]{Bra83} showed that $h$ maps cubefree words to
cubefree words.  Moreover, since $h$ is uniform and injective, the set of
words $h(yy)$ consists of at least $2^{m/2}$ cubefree squares of length
$12m$.  Asymptotically, we thus have exponentially many cubefree
binary squares of length $n$, as required.
\end{proof}

We now prove Theorem~\ref{main}.

\begin{proof}[Proof of Theorem~\ref{main}]
In the proof of Proposition~\ref{exp_many} we showed that there are
at least $2^{m/2}$ cubefree binary squares of length $12m$ for every
positive integer $m$.  Let $S$ therefore be any set of cubefree squares
over $\{0,1\}$ where $S$ contains at least $2^{m/2}$ words of length $12m$
for every positive integer $m$.  Let ${\bf x} = x_1x_2 \cdots $
be any infinite cubefree binary word over $\{2,3\}$.  Construct a word
\[
{\bf w} = x_1 S_1 x_2 S_2 \cdots,
\]
where the set of $S_i$'s is equal to the set $S$, so that ${\bf w}$ is
cubefree and contains exponentially many distinct squares of length $n$.
Let $g$ be the morphism
\begin{eqnarray*}
0 & \to & 001001101 \\
1 & \to & 001010011 \\
2 & \to & 001101011 \\
3 & \to & 011001011.
\end{eqnarray*}
Brandenburg \cite[Theorem~6]{Bra83} showed that $g$ maps cubefree words to
cubefree words.  Thus, $g({\bf w})$ is cubefree and, by the uniformity and
injectivity of $g$, contains exponentially many distinct squares of length $n$.
\end{proof}

Note that Theorem~\ref{main} implies that existence of an infinite cubefree
binary word with exponential \emph{factor complexity}---i.e., with
exponentially many factors of length $n$.  Similarly, one can easily
construct an infinite squarefree word over $\{0,1,2\}$ with exponential
factor complexity.

\begin{proposition}
\label{cmplxty}
There exists an infinite squarefree word over $\{0,1,2\}$ with
exponential factor complexity.
\end{proposition}

\begin{proof}
Let ${\bf w}$ be any infinite squarefree word over $\{0,1,2\}$
and let ${\bf x}$ be any infinite word over $\{3,4\}$ with $2^n$ factors
of length $n$ for every positive $n$.  Let ${\bf y}$ be the word obtained
by forming the \emph{perfect shuffle} of ${\bf w}$ and ${\bf x}$: that is,
if ${\bf w} = w_0w_1w_2 \cdots$ and ${\bf x} = x_0x_1x_2 \cdots$, then
define ${\bf y} = w_0x_0w_1x_1w_2x_2 \cdots$.  Clearly,
${\bf y}$ is a squarefree word with exponential factor complexity.
Let $f$ be the morphism
\begin{eqnarray*}
0 & \to & 010201202101210212 \\
1 & \to & 010201202102010212 \\
2 & \to & 010201202120121012 \\
3 & \to & 010201210201021012 \\
4 & \to & 010201210212021012.
\end{eqnarray*}
Brandenburg \cite[Theorem~4]{Bra83} showed that $f$ maps squarefree words to
squarefree words.  The uniformity and injectivity of $f$ implies that
$f({\bf y})$ is a squarefree word with exponential factor complexity,
as required.
\end{proof}

\bibliographystyle{amsalpha}

\end{document}